\documentclass[11pt]{amsart}

\usepackage{latexsym}
\usepackage{amssymb}
\usepackage{amsmath}
\usepackage{enumerate}
\usepackage{color}
\usepackage{comment}
\usepackage{a4wide}
\newtheorem{theorem}{Theorem}[section]
\newtheorem{lemma}[theorem]{Lemma}
\newtheorem{proposition}[theorem]{Proposition}
\newtheorem{corollary}[theorem]{Corollary}
\newtheorem{definition}[theorem]{Definition}

\newtheorem{remark}[theorem]{Remark}

\newcommand\supp{\mathop{\rm supp}}

\newcommand\tr{\mathop{\rm tr}}

\newcommand\bist{\mathop{\mathsf{C}}}
\newcommand\ball{\mathop{\rm Ball}}
\newcommand\cnm{\mathop \mathbb{C}^n\otimes\mathbb{C}^m}
\newcommand\Tr{\mathop{\rm Tr}}

\newcommand\eps{\varepsilon}

\newcommand\ot{\otimes}

\newcommand{\cl}[1]{\mathcal{#1}}
\newcommand{\bb}[1]{\mathbb{#1}}

\begin{document}

\title{Coupling capacity in C*-algebras}

\author[A. Skalski]{Adam Skalski}
\address{Institute of Mathematics of the Polish Academy of Sciences, ul. \'Sniadeckich 8, 00-656 Warszawa, Poland}
\email{a.skalski@impan.pl}

\author[I. G. Todorov]{Ivan G. Todorov}
\address{
School of Mathematical Sciences, University of Delaware, 501 Ewing Hall,
Newark, DE 19716, USA}
\email{todorov@udel.edu}

\author[L. Turowska]{Lyudmila Turowska}
\address{Department of Mathematical Sciences, Chalmers University
of Technology and the University of Gothenburg, Gothenburg SE-412 96, Sweden}
\email{turowska@chalmers.se}

\date{18 February 2023}

\maketitle


\begin{abstract}
Given two unital C*-algebras equipped with states and a positive operator in 
the enveloping von Neumann algebra of their minimal tensor product, we 
define three parameters that measure the capacity of the operator to align with a coupling 
of the two given states. Further we establish a duality formula that shows the equality of two of the parameters for operators in the minimal tensor product of the relevant C*-algebras.
In the context of  abelian C*-algebras our parameters are related to  quantitative versions of Arveson's Null Set Theorem and to dualities considered in the theory of optimal transport. On the other hand,   restricting to matrix algebras we recover and generalise quantum versions of Strassen's Theorem. We show that in the latter case our parameters can detect maximal entanglement and separability. 
\end{abstract}


\section{Introduction}\label{s_intro}

Strassen's Theorem \cite{strassen} characterising the existence of a probability measure on a product measurable space,
having fixed marginals and prescribed support, has enjoyed an illustrious history, both leading to new fruitful 
research directions and
having significant applications.
Such joint probability measures, known as couplings of the pair of original measures, 
are the starting point of the theory of optimal transport and
appear as a fundamental concept in the celebrated Monge-Kantorovich duality \cite{villani0, villani}. 
They are also at the heart of Arveson's Null Set Theorem \cite{a}, which formed the base of vast parts of 
non-selfadjoint operator algebra theory and had a lasting impact on the study of invariant spaces 
for collections of Hilbert space operators (see \cite{dav}). 
Arveson's Null Set Theorem was given a quantitative formulation by Haydon and Shulman \cite{hs}; 
the quantifying parameters defined therein 
were shown in \cite{hs} to be capacities in the sense of Choquet's capacitability theory 
\cite{choquet}.

Recently, a quantum version of Strassen's Theorem was established \cite{zyyy}, inspired by applications 
to quantum information theory. In the latter setting, the result identifies necessary and sufficient conditions
for the existence of a state on the tensor product of two matrix algebras with prescribed 
marginal states. A study of related phenomena in the case of infinite dimensional type I factors was pursued in \cite{fgz}. 


The aim of the present paper is to formulate and exploit 
a common framework that unifies and extends the several aforementioned themes.
Given two unital C*-algebras $\cl A$ and $\cl B$, equipped with respective states $\phi$ and $\psi$, 
we introduce three parameters that measure the capacity that
the couplings of $\phi$ and $\psi$ -- 
that is, states on the minimal tensor product $\cl A\otimes \cl B$ whose marginals coincide with $\phi$ and 
$\psi$, respectively -- align with a given positive operator $T$ in the enveloping von Neumann algebra
$(\cl A\otimes\cl B)^{**}$. 
In the case $T$ is an orthogonal projection, 
these parameters can be thought of as capacities of that projection to support a
quantum coupling of the two given states.
We establish a duality result of Monge-Kantorovich type in this context, stating that 
two of the introduced parameters coincide whenever $T \in \cl A \otimes \cl B$ (see Theorem \ref{th_bg}), 
and are bounded from above by the third.

Restricting to abelian C*-algebras and to orthogonal projections, 
we show that our parameters coincide with the 
Choquet capacities of Haydon and Shulman  (see \cite{hs}). 
The positive operator $T\in (\cl A\otimes\cl B)^{**}$ can in this case be thought of as a measurable
cost function in the sense of the theory of optimal transport \cite{villani0}. 
On the other hand, restricting to the case where the C*-algebras are matrix algebras, 
we see that the duality result implies 
the quantum versions of Strassen's Theorem established in \cite{zyyy, fgz}. 
Thus, our result can  be thought of as a quantitative extension of a
C*-algebra version of Strassen's Theorem, closely related to 
a non-commutative  version of Arveson's Null Set Theorem.

We show that, in the case of matrix algebras, 
the introduced coupling capacities
can detect maximal entanglement and separability of bipartite states
(see Theorem \ref{p_sep}). 
We further establish several general facts, showing 
that our parameters enjoy natural continuity properties, both when considered as functions on 
the positive operator in $\cl A\otimes\cl B$, and on the pair $(\phi,\psi)$ of states. Finally we would like to note that in recent years the (noncommutative) optimal transport techniques appeared in the operator algebraic contexts ranging from the classification theory of C*-algebras \cite{jsv} to free probability \cite{gjns}.

The detailed plan of the paper is as follows: after describing the basic notation in the remainder of the introduction, in Section 2 we introduce our capacities, establish the relationship between them (notably in Theorem \ref{th_bg}) and study the relevant continuity properties. Here we also  discuss the commutative case, providing the connection to Arveson's Null Set Theorem and to the classical  Monge-Kantorovich duality. Finally in Section 3 we focus on the matrix case, deducing the quantum Strassen's theorem of \cite{zyyy} from our main results, developing the connection between coupling capacities and entanglement, and discussing several examples.

\smallskip


We finish this section with setting notation. 
For a C*-algebra $\cl A$, we denote by 
$\cl A_h$ the real vector space of all hermitian elements in $\cl A$, 
by $\cl A_+$ the cone of its positive elements, and 
by $\cl P(\cl A)$ the set of all projections in $\cl A$; note that when $\cl A$ is a von Neumann algebra,
$\cl P(\cl A)$ is a complete ortho-lattice. We use standard notation for the
supremum ($\vee$) and infimum ($\wedge$) in $\cl P(\cl A)$.
We denote by 
$\cl A^*$ the dual of $\cl A$, by $\cl A^*_+$ the positive functionals on $\cl A$,
and by $\cl A^{**}$ the second dual of $\cl A$. We view $\cl A^{**}$ as
the enveloping von Neumann algebra of $\cl A$, and $\cl A$ as a C*-subalgebra
of $\cl A^{**}$.
If $\phi\in \cl A^*$ then $\phi$ has a unique extension to a
weak* continuous functional on $\cl A^{**}$, which will be denoted by the same symbol; this operation preserves the property of being a state.

All C*-algebras considered in the paper will be unital; the unit of a C*-algebra $\cl A$ will be denoted by $1_{\cl A}$
(or $1$ if there is no danger of confusion). 
An \emph{operator system} in a C*-algebra $\cl A$ is a selfadjoint (and not necessarily closed) 
linear subspace of $\cl A$ containing $1_{\cl A}$. 
A \emph{state} of an operator system $\cl S$ is a positive functional $f : \cl S\to \bb{C}$ such that $f(1_{\cl A}) = 1$; 
the (convex) set of all states of $\cl S$ is denoted by $S(\cl S)$. 

We write $M_n$ for the algebra of all $n$ by $n$ matrices, and ${\rm tr}$ (resp. ${\rm Tr}$) 
for the normalised (resp.\ taking value $1$ on minimal projections) trace on $M_n$. 
If we want to emphasise the underlying dimension, we write  $\tr_n$. 
We let $(\epsilon_{i,j})_{i,j=1}^n$ be the canonical matrix unit system in $M_n$. 
Given a state $\omega : M_n\to \bb{C}$, there exists a unique positive semi-definite matrix $A_{\omega}$
with $\tr(A_{\omega}) = 1$ 
(called the density matrix of $\omega$)
such that $\omega(B) = \tr(A_{\omega}B)$, $B\in M_n$. 
We will sometimes identify $\omega$ with $A_{\omega}$. 
In the lack of preferred matrix unit system inside $M_n$, we will use the notation $\cl L(\bb{C}^n)$. 
Given vectors $\xi$ and $\eta$, we use the notation $\xi\eta^*$ for the rank one operator given by 
$(\xi\eta^*)(\zeta) = \langle \zeta,\eta\rangle \xi$. Note that the scalar products are linear on the left.

If $X$ is a compact Hausdorff space, we denote as usual by $C(X)$ the (abelian) C*-algebra of all 
continuous complex-valued functions on $X$ and by $M(X)$ the space of all complex Borel measures on $X$. 
Note that, by the
Riesz Representation Theorem, $M(X)$ can be canonically identified with $C(X)^*$.


\section{Definition of coupling capacities and their fundamental properties} \label{s_cc}

In this section, we define three parameters that form the focus of the paper
and examine some of their properties. The main result of the section is
Theorem \ref{th_bg}, which can be thought of as a non-commutative Monge-Kantorovich type duality.


\subsection{Definitions} \label{ss_def}

Let $\cl A$ and $\cl B$ be unital C*-algebras, equipped with states $\phi$ and $\psi$,
respectively.
We denote by $\cl A\otimes \cl B$ (resp. $\cl A\odot \cl B$) 
the minimal (resp. the algebraic) tensor product of $\cl A$ and $\cl B$.
For an element $\sigma\in (\cl A\otimes\cl B)^*$, we denote by
$\sigma_{\cl A}$ (resp. $\sigma_{\cl B}$) the
element of $\cl A^*$ (resp. $\cl B^*$) given by
$$\sigma_{\cl A}(a) = \sigma(a\otimes 1) \ \ \ \ (\mbox{resp. } \sigma_{\cl B}(b) = \sigma(1\otimes b));$$
thus, $\sigma_{\cl A}$ (resp. $\sigma_{\cl B}$) is the $\cl A$-marginal
(resp. the $\cl B$-marginal) of $\sigma$.

\begin{definition}\label{d_coup}
A positive functional $\sigma : \cl A\otimes \cl B\to \bb{C}$ is called a 
\emph{coupling} of the states $\phi$ and $\psi$ (or a \emph{$(\phi,\psi)$-coupling}) if 
$\sigma_{\cl A} = \phi$ and $\sigma_{\cl B} = \psi$.
\end{definition}

\noindent 
We denote by $\bist(\phi,\psi)$ the set of all $(\phi,\psi)$-couplings. Note that each $(\phi,\psi)$-coupling is automatically a state and that $\bist(\phi,\psi)$, equipped with weak$^*$ topology, is a compact convex set.

\medskip

\noindent {\bf Remarks. (i) }
Suppose that $X$ (resp. $Y$) is a compact Hausdorff space, $\cl A = C(X)$ (resp. $\cl B = C(Y)$), and let 
$\mu$ (resp.\ $\nu$) be a Borel probability measure on $X$ (resp. $Y$). 
Viewing $\mu$ (resp. $\nu$) as a state on $\cl A$ (resp.\ $\cl B)$, we see that the elements of 
$\bist(\mu,\nu)$ are precisely the couplings of the measures $\mu$ and $\nu$ in terms of 
the theory of optimal transport (see \cite[Definition 1.1]{villani}). 

\smallskip

{\bf (ii) } 
Specialising further, let $\cl A$ and $\cl B$ coincide with the algebra $\cl D_n$ of all diagonal matrices in $M_n$
(where $n\in \bb{N}$). 
Recall that a matrix $\Lambda = (\lambda_{i,j})_{i,j=1}^n \in M_n$ is called \emph{bistochastic} if 
$$\lambda_{i,j}\geq 0 \ \mbox{ and } \ \sum_{j'=1}^n \lambda_{i,j'} = \sum_{i'=1}^n \lambda_{i',j} = 1, \ \ \ i,j = 1,\dots,n.$$
In view of the canonical (algebraic) identification $\cl D_n\otimes \cl D_n \equiv M_n$, we can thus refer to an 
element of $\cl D_n\otimes \cl D_n$ being bistochastic. 
If $\sigma \in (\cl D_n\otimes\cl D_n)^*$, there exists a (unique) $A_{\sigma}\in \cl D_n\otimes\cl D_n$ such that 
$$\sigma(T) = {\rm tr}_{n^2}(TA_{\sigma}), \ \ \ T\in \cl D_n\otimes\cl D_n.$$
It is straightforward to verify that $\sigma\in \bist({\rm tr},{\rm tr})$ if and only if the matrix 
$\frac{1}{n} A_\sigma$ is bistochastic.

\medskip

Let $\cl A$ and $\cl B$ be unital C*-algebras.
We have that $\cl A \otimes 1\subseteq \cl A\otimes \cl B$ as C*-algebras, and hence 
$$\cl A^{**}\otimes 1 = (\cl A\otimes 1)^{**}
\subseteq (\cl A\otimes \cl B)^{**}$$
as von Neumann algebras. Similarly, 
$1\otimes \cl B^{**}\subseteq (\cl A\otimes\cl B)^{**}$. 
By \cite[Proposition 9.2.1]{bo}, 
the two von Neumann subalgebras 
$\cl A^{**}\otimes 1$ and $1\otimes \cl B^{**}$ of 
$(\cl A\otimes\cl B)^{**}$ mutually commute and 
there exists a canonical separately weak* continuous
embedding 
$$\cl A^{**}\odot \cl B^{**} \subseteq (\cl A\otimes\cl B)^{**}.$$
In particular, we can consider 
$\cl A^{**}\otimes 1 + 1\otimes \cl B^{**}$ as
an operator subsystem of 
$(\cl A\otimes\cl B)^{**}$. 
The latter identification will be made throughout the rest of the paper.

For a unital C*-algebra $\cl A$ and a state $\phi$ on $\cl A$, we will refer to the pair $(\cl A,\phi)$ as a \emph{measured C*-algebra}. 
(The motivation for the terminology comes from the commutative case $\cl A = C(X)$, where $X$ is a compact Hausdorff space and the fact that, in this case, states on $\cl A$ correspond canonically to Borel probability measures on $X$.)

\begin{definition}\label{def:capac}
Let $(\cl A, \phi)$ and $(\cl B, \psi)$ be measured $C^*$-algebras.
For $T\in\cl (\cl A\otimes\cl B)^{**}_+$ with $\|T\|\leq 1$, let 
$$\alpha(T) = \sup\{\sigma(T) : \sigma\in \bist(\phi,\psi)\},$$
\begin{equation}\label{eq_beta}
\beta(T) = \inf\{\phi(a) + \psi(b) : a\in \cl A^{**}_+, b\in \cl B^{**}_+, \ T\leq a\otimes 1 + 1\otimes b\},
\end{equation}
and 
$$\gamma(T) = \inf\{\phi(p) + \psi(q) : p\in\cl P(\cl A^{**}), q\in\cl P(\cl B^{**}), T\leq (p\otimes 1)\vee (1\otimes q)\}.$$
\end{definition}

We will refer to $\alpha(T)$ (resp. $\gamma(T)$) 
as the \emph{coupling capacity} (resp. the \emph{projective coupling capacity})
of $(\phi,\psi)$ with respect to $T$.

\begin{remark}\label{r_less}
\rm
{\bf (i)} 
By the compactness of the set $\bist(\phi,\psi)$ in the weak* topology, 
the supremum in the definition of $\alpha(T)$ is achieved if $T \in \cl A \otimes \cl B$.

\smallskip

{\bf (ii)}
Let 
$$\tilde{\bist}(\phi,\psi) = \left\{\sigma\in (\cl A\otimes\cl B)^*_+ : 
\sigma_{\cl A}\leq \phi \mbox{ and } \sigma_{\cl B}\leq \psi\right\}.$$
For $T \in\cl (\cl A\otimes\cl B)^{**}_+$, we have that
\begin{equation}\label{eq_biso}
\alpha(T) = \sup\left\{\sigma(T) : \sigma\in \tilde{\bist}(\phi,\psi)\right\}.
\end{equation}
Indeed, letting $\alpha'(T)$ denote the right hand side of (\ref{eq_biso}), we trivially have 
$\alpha(T)\leq \alpha'(T)$. 
Suppose that $\sigma\in \tilde{\bist}(\phi,\psi)$. 
Then $\sigma(1) = \sigma_{\cl A}(1) \leq \phi(1) = 1$.
Let $\phi' = \phi - \sigma_{\cl A}$ and $\psi' = \psi - \sigma_{\cl B}$; then 
$\phi'$ and $\psi'$ are positive functionals. 
If $\sigma(1) = 1$ then 
$$\phi'(1) = \phi(1) - \sigma(1\otimes 1) = 0$$
and hence $\phi' = 0$, that is, $\sigma_{\cl A} = \phi$;
similarly, $\sigma_{\cl B} = \psi$, that is, $\sigma\in \bist(\phi,\psi)$. 
We may hence assume that $\sigma(1) < 1$. 
Set
$$\sigma' = \sigma + \frac{1}{1 - \sigma(1)}\phi'\otimes \psi';$$
for $a\in \cl A$ we then have 
$$\sigma'(a\otimes 1) 
= 
\sigma(a\otimes 1) + \frac{1}{1 - \sigma(1)}\phi'(a)\psi'(1)
= 
\sigma(a\otimes 1) + \phi'(a) = \phi(a),$$
that is, $\sigma'_{\cl A} = \phi$; similarly, $\sigma'_{\cl B} = \psi$, that is, $\sigma'\in \bist(\phi,\psi)$. 
In addition, $\sigma \leq \sigma'$ and hence 
$\alpha'(T)\leq \alpha(T)$, establishing (\ref{eq_biso}).
\end{remark}



\subsection{A Monge-Kantorovich type duality} \label{ss_monge}

The purpose of this subsection is to identify the relations between the parameters $\alpha$, $\beta$ and 
$\gamma$. 
As a motivating example, consider the special case where 
$\cl A = \cl B = \cl D_n$, equipped with normalised trace $\tr$. 
As pointed out in Remark (ii) after Definition \ref{d_coup}, up to rescaling,
the elements in $\bist(\tr, \tr)$ correspond to bistochastic matrices.
Using the Birkhoff-von Neumann Theorem, it is straightforward to see that 
$\alpha(E) = \gamma(E)$ for every projection $E$ in $\cl D_n \otimes \cl D_n$. 
(In fact, one can easily verify that 
both $\alpha(E)$ and $\gamma(E)$ are equal to the normalised length of a maximal partial graph of a (partial) bijection, contained in $E$.)

We begin with a general min-max result regarding the state extensions. After the first version of this article was announced, Michael Hartz kindly pointed out to us that a very similar result is contained in \cite[Proposition 6.2]{a2}.

\begin{lemma}\label{l_ex}
Let $\cl C$ be a unital C*-algebra and 
$\cl S \subseteq \cl C$ be an operator subsystem. 
For $\tau \in S(\cl S)$, let ${\rm Ext}(\tau) = \left\{\omega \in S(\cl C): \omega|_{\cl S} = \tau\right\}$. 
Then, for any hermitian element $x\in \cl C$, we have 
\begin{equation}\label{eq_ome}
\sup \left\{\omega(x) : \omega \in {\rm Ext}(\tau)\right\} = \inf \{\tau(y) : y \in \cl S_h, y \geq x\}.
\end{equation}
\end{lemma}

\begin{proof}
Let $t_0$ (resp. $t$) denote the left (resp. right) hand side of (\ref{eq_ome}). 
If $y \in \cl S_h$, $x \leq y$, and $\omega \in {\rm Ext}(\tau)$, 
then $\omega(x)\leq \omega(y) = \tau(y)$, so 
$t_0\leq t$. 

If $x\in \cl S$, then both sides of (\ref{eq_ome}) are equal to 
$\tau(x)$, so we may assume that $x \notin \cl S$.
Consider the subspace $\cl T := \cl S + \mathbb{C}x$ and define a linear functional 
$\tau' : \cl T \to \mathbb{C}$ by letting
\[ \tau'(y + \lambda x) = \tau(y) + \lambda t, \ \ \  y \in \cl S, \ \lambda \in \mathbb{C}.\]
The fact that $\tau'$ is well-defined is a consequence of the fact that $x \notin \cl S$; in addition, $\tau'$ is clearly unital. 

Suppose that $z = y + \lambda x \in \cl T_+$; then $\lambda\in \bb{R}$ and $y\in \cl S_h$. 
We will show that $\tau'(z) \geq 0$. Assume first that $\lambda < 0$. 
Then $(-\lambda)^{-1} y \geq x$, so $\tau((-\lambda)^{-1} y) \geq t$, and hence 
$$\tau'(z) = \tau(y) + \lambda t = -\lambda\left(\tau((-\lambda)^{-1} y) - t\right) \geq 0.$$
If $\lambda > 0$ then, for any $\omega \in {\rm Ext}(\tau)$, we have that 
$$\tau\left(\frac{y}{\lambda}\right) + \omega(x) = \omega\left(\frac{y}{\lambda}\right) + \omega(x) 
= \frac{1}{\lambda} \omega(z) \geq 0.$$ 
Thus
$\tau(\frac{y}{\lambda}) + t_0 \geq 0$. By the first part of the proof, 
$\tau(\frac{y}{\lambda}) + t \geq 0$; this implies that $\tau'(z)\geq 0$. 
Finally, for $\lambda = 0$ the fact that $\tau'(z) \geq 0$ is trivial. 
Thus $\tau'$ is a positive functional. Extend $\tau'$ to a state $\widetilde{\tau}$ on $\cl C$. 
We have that $\widetilde{\tau}\in {\rm Ext}(\tau)$ and that $\widetilde{\tau}(x) = t$. It follows that $t\leq t_0$, completing the proof.
\end{proof}

\begin{lemma}\label{l_spand}
Let $H$ be a Hilbert space and $P$ and $Q$ be projections on $H$. 
\begin{itemize}
\item[(i)]
If $PQ = QP$ and 
$T$ is a positive contraction on $H$, then
$T\leq P + Q$ if and only if $T\leq P \vee Q$.

\item[(ii)]
If $rP \leq Q$ for some $r > 0$ then $P\leq Q$. 
\end{itemize}
\end{lemma}

\begin{proof}
(i) 
Assume that $T\leq P + Q$,
suppose that $T^{1/2} (P\vee Q)^{\perp} \neq 0$ and let $\xi\in H$ be such that
$T^{1/2}(P^{\perp}Q^{\perp})\xi \neq 0$. Set $\eta = (P^{\perp}Q^{\perp})\xi$;
then $(P\eta,\eta) = (Q\eta,\eta) = 0$ but
$0\neq \|T^{1/2}\eta\|^2 = (T\eta,\eta)$, a contradiction with the assumption that
$T\leq P + Q$.
It follows that $T^{1/2}(P\vee Q)^{\perp} = 0$ and hence 
$(P\vee Q)^{\perp}T = 0$, implying that ${\rm ran}(T)\subseteq {\rm ran}(P\vee Q)$. 
Since $\|T\|\leq 1$, the latter condition implies $T\leq P\vee Q$. 
The converse implication follows from the fact that $P\vee Q\leq P + Q$. 

(ii) 
For $\xi\in H$, we have
$$r\|PQ^{\perp}\xi\|^2
= (rPQ^{\perp}\xi,Q^{\perp}\xi)  \leq (QQ^{\perp}\xi,Q^{\perp}\xi) = 0,$$
showing that $PQ^{\perp} = 0$. Thus, $P\leq Q$.
\end{proof}

The second part of the following theorem is one of the key results of the paper.

\begin{theorem}\label{th_bg}
Let $(\cl A, \phi)$ and $(\cl B, \psi)$ be measured $C^*$-algebras and
$T\in (\cl A\otimes\cl B)^{**}$ be a positive contraction. Then 
$$\alpha(T) \leq \beta(T) \leq \gamma(T) \leq 1.$$
Furthermore if $T \in (\cl A \otimes \cl B)_+$ then $\alpha(T) = \beta(T)$.
\end{theorem}

\begin{proof}
Let $T \in (\cl A\otimes\cl B)^{**}$ be as above. It is easy to see that if $\sigma \in \bist(\phi,\psi)$ and $a\in \cl{A}^{**}$, $b \in \cl{B}^{**}$, then
$\sigma(a \otimes 1 + 1 \otimes b) = \phi(a) + \psi(b)$. This immediately shows that $\sigma(T) \leq \phi(a) + \psi(b)$ whenever $T \leq a \otimes 1 + 1 \otimes b$, so that $\alpha(T) \leq \beta(T)$.

Restricting the right hand side of (\ref{eq_sae}) to projections $p = a$ and $q = b$, 
an application of Lemma \ref{l_spand} (i) shows that  $\beta(T)\leq \gamma(T)$.
Finally, since $p = 1$, $q = 0$ gives a feasible choice for the projections $p$ and $q$, 
we have that $\gamma(T) \leq 1$. 

Assume now that $T \in (\cl A \otimes \cl B)_+$ and in Lemma \ref{l_ex} set $\cl C :=\cl A\otimes \cl B$ and
$\cl S := \cl A\otimes 1 + 1\otimes \cl B$, equipped with the state 
$\tau := \phi\otimes\psi|_{\cl S}$. 
Note that if $a\in \cl A$ then $a\otimes 1\in \cl S$ and
$\tau(a\otimes 1) = \phi(a)$; similarly, if $b\in \cl B$ then $\tau(1\otimes b) = \psi(b)$.
By Lemma \ref{l_ex}, 
\begin{eqnarray}\label{eq_sae}
	&& \alpha(T) 
	=
	\inf\{\phi(a) + \psi(b) : a\otimes 1 + 1\otimes b \in 
	(\cl A\otimes 1 + 1\otimes \cl B)_h,\\
	& & 
	\hspace{7.8cm} 
	T\leq a\otimes 1 + 1\otimes b\}.\nonumber
\end{eqnarray}
The condition 
$a\otimes 1 + 1\otimes b \in (\cl A\otimes 1 + 1\otimes \cl B)_h$ implies that 
$a\otimes 1 + 1\otimes b = \frac{a+a^*}{2} \ot 1 + 1 \ot \frac{b+b^*}{2}$ and therefore that 
\begin{eqnarray*}
	\phi(a) + \psi(b) 
	& = & 
	\tau(a\otimes 1 + 1\otimes b) 
	= 
	\tau\left(\frac{a+a^*}{2} \ot 1 + 1 \ot \frac{b+b^*}{2}\right)\\
	& = & 
	\phi\left(\frac{a+a^*}{2}\right) + \psi\left(\frac{b+b^*}{2}\right).
\end{eqnarray*}
It follows that the elements $a$ and $b$ on the right hand side of (\ref{eq_sae}) can be assumed hermitian. 

Assume that $a\in\cl A_h$ and $b\in\cl B_h$ are such that
$T\leq a\otimes 1 + 1\otimes b$.
We claim that, without loss of generality, the elements $a$ and $b$ can be assumed positive. 
Write ${\rm sp}(x)$ for the spectrum of $x$ and let 
$s = \min{\rm sp(a)}$ and $t = \min{\rm sp(b)}$.
If $\min\{s,t\} \geq 0$ then $a$ and $b$ are positive and the claim is vacuous. 
Suppose that $\min\{s,t\} < 0$, say $s < 0$. 
Let $(\xi_n)_{n\in\bb{N}}$ (resp. $(\eta_n)_{n\in\bb{N}}$) be a 
sequence of unit vectors in the Hilbert space $H_{\cl A}$ (resp. $H_{\cl B}$) of the faithful representation of $\cl A$
(resp. $\cl B$) such that 
$$s = \lim_{n\to \infty} \langle a\xi_n,\xi_n\rangle \ \ (\mbox{resp. } 
t = \lim_{n\to \infty} \langle b\eta_n,\eta_n\rangle).$$
As 
\begin{eqnarray*}
	0 
	& \leq & 
	\langle T(\xi_n\otimes\eta_n),\xi_n\otimes\eta_n\rangle\\
	& \leq &  
	\langle (a\otimes 1)(\xi_n\otimes\eta_n),\xi_n\otimes\eta_n\rangle + \langle (1\otimes b)(\xi_n\otimes\eta_n),\xi_n\otimes\eta_n\rangle\\
	& = & 
	\langle a\xi_n,\xi_n\rangle + \langle b\eta_n,\eta_n\rangle,
\end{eqnarray*}
we have that $s + t\geq 0$. 
Let $a' = a - s1$ and $b' = b + s1$. Then $a'\geq 0$ and
$b' \geq b - t1 \geq 0$. 
On the other hand, trivially, 
$$a\otimes 1 + 1\otimes b = a'\otimes 1 + 1\otimes b' \ \mbox{ and } \ 
\phi(a) + \psi(b) = \phi(a') + \psi(b').$$ 
We have shown that the elements $a$ and $b$ in (\ref{eq_sae}) can be assumed to be positive, and 
combined with the first paragraph, this implies that 
$\alpha(T) = \beta(T)$.

\end{proof}

\begin{remark}
\rm
It is natural to ask whether the equality $\alpha(T) = \beta(T)$ can be extended beyond elements of $(\cl A \otimes \cl B)_+$. An instance where this is true can be seen in Proposition \ref{prop:lowerscts} below; the general case remains open.	
\end{remark}

It is clear that if $T\in (\cl A\otimes\cl B)^{**}$ is a positive contraction and $E$ is its range projection then 
$\gamma(T) = \gamma(E)$. It is therefore natural to restrict attention to the values of the parameter $\gamma$
on the projections in $(\cl A\otimes\cl B)^{**}$ alone.  
As we next note, the inequality $\beta(T) \leq \gamma(T)$ can be strict even for $T \in \cl A \otimes \cl B$. 
We will need a special case of the following proposition which, at the same time, exhibits a 
case, where an equality between $\beta$ and $\gamma$ takes place.

\begin{proposition}\label{p_papb}
Let $(\cl A,\phi)$ and $(\cl B,\psi)$ be measured C*-algebras, and 
$e\in \cl P(\cl A^{**})$ and $f\in \cl P(\cl B^{**})$. Then 
\begin{equation}\label{eq_genef}
\gamma(e\otimes f) = \min\{\phi(e),\psi(f)\}.
\end{equation}
If, in addition, $\cl A = M_n$ and $\cl B = M_m$ for some $n, m \in \mathbb{N}$, 
and $\phi = \tr_n$ and $\psi = \tr_m$, then 
\begin{equation}\label{eq_finef}
\beta(e\otimes f) = \gamma(e\otimes f) = \min\{\phi(e),\psi(f)\}.
\end{equation}
\end{proposition}

\begin{proof}
Since $e\otimes f \leq e \otimes 1$ and $e\otimes f\leq 1 \otimes f$,
we have that $\gamma(e\otimes f) \leq \min\{\phi(e),\psi(f)\}$.
On the other hand, assume that $p$ and $q$ are projections with
$e\otimes f\leq (p\otimes 1)\vee (1\otimes q)$. Then
$(e\otimes f)(p^{\perp}\otimes q^{\perp}) = 0$ and hence
either $e\leq p$ or $f\leq q$. This implies that $\phi(p) + \psi(q) \geq \min\{\phi(e),\psi(f)\}$
and (\ref{eq_genef}) is established. 

Proceeding to the justification of (\ref{eq_finef}), 
in view of Theorem \ref{th_bg} and (\ref{eq_genef}), it suffices to show that 
$\min\{\phi(e),\psi(f)\} \leq \alpha(e\otimes f)$. 
Choose orthonormal bases $(e_1,\ldots,e_k)$ (resp. $(f_1,\ldots,f_l)$) of 
the ranges of $e$ (resp. $f$), 
and complete it to a basis of $\mathbb{C}^n$ (resp. $\mathbb{C}^m$). 
Assume, say, that $\frac{k}{n} \leq \frac{l}{m}$. 
Let $\pi$ be the probability distribution on $\{1,\ldots, n\}\times \{1,\ldots, m\}$, given by 
$$
\pi(i,j) = 
\begin{cases}
\frac{1}{nl} & \text{if } i\leq k \mbox{ and } j\leq l,\\
0 & \text{if } i\leq k \mbox{ and }  j > l, \\
\frac{1}{l(n-k)}\left(\frac{l}{m}-\frac{k}{n}\right) & \text{if } i > k \mbox{ and } j\leq l,\\
\frac{1}{m(n-k)} & \text{if } i > k \mbox{ and } j > l.
\end{cases}
$$
Then the marginals of $\pi$ coincide with the uniform distributions and 
$$\pi(\{1,\ldots, k\}\times \{1,\ldots, l\}) = \frac{k}{n}.$$
Let
\[ D = \sum_{i=1}^n \sum_{j=1}^m \pi(i,j) e_ie_i^* \ot f_jf_j^*.\]
It is then easy to check that the state on $M_n\otimes M_m$ 
with density matrix $D$ belongs to $\mathsf{C}(\tr_n, \tr_m)$, and 
$\tr (D(e\otimes f)) = \frac{k}{n}$. Thus $\min\{\phi(e),\psi(f)\}=\frac{k}{n} \leq \alpha(e \ot f)$ 
and the proof is complete.
\end{proof}

\begin{remark}\label{r_gammadif}
\rm 
The inequality $\beta(E) \leq \gamma(E)$ in Theorem \ref{th_bg}, for projections $E\in  \cl A\otimes\cl B$, 
can be strict. 
Indeed, let $\cl A = \cl B = M_2$ and 
$\phi = \psi$ coincide with the vector state $\omega_{\xi}$
corresponding to the vector $\xi=\frac{1}{\sqrt{2}}\begin{pmatrix}1\\1\end{pmatrix}$.
Let $p$ be the rank one projection with range the subspace generated by the 
vector $e_1=\begin{pmatrix}1\\0\end{pmatrix}$, and $E = p\otimes p$.
By Proposition \ref{p_papb}, $\gamma(E) = \frac{1}{2}$. 

We claim that $\beta(E)= \frac{1}{4}$.
Indeed, suppose that $\omega \in S(M_2 \ot M_2)$ is an element of $\tilde{\bist}(\omega_\xi,\omega_\xi)$.
Writing $p_\xi$ for the projection onto $\bb{C}\xi$, 
this means that $\omega(p_{\xi}^\perp\ot 1) = \omega(1 \ot p_\xi^\perp) = 0$.
Now an elementary calculation shows that the density matrix $A_{\omega}$ of $\omega$
has rank one. 
Since ${\rm tr}(A_{\omega})=1$, we conclude that $\omega = \omega_\xi \ot \omega_\xi$. 
This implies that $\alpha(E) = \omega_\xi (p)^2= \frac{1}{4}$.

The state $\omega_\xi$ above is not faithful, but considering instead of this a state of the form 
$(1- \eps) \omega_\xi + \eps \omega_\eta$, say with $\eta = \frac{1}{\sqrt{2}}\begin{pmatrix}1\\-1\end{pmatrix}$, and using once again Proposition \ref{p_papb} and the upper semicontinuity of Proposition \ref{prop:upper} (i) we can see that the inequality $\beta(E) \leq \gamma(E)$ may be strict even for faithful states.
\end{remark}

We next exhibit another situation, where 
an equality between the parameters $\beta$ and $\gamma$ takes place.

\begin{proposition}\label{p_ifcom}
Let $E\in (\cl A\otimes \cl B)^{**}$ be a projection. Suppose that the infimum in the definition of $\beta$ 
(see (\ref{eq_beta})) is achieved at a pair $(a,b)$ such that $E(a\otimes 1) = (a\otimes 1)E$ and 
$E(1\otimes b) = (1\otimes b)E$. Then $\beta(E) = \gamma(E)$. 
\end{proposition}

\begin{proof}
Let $\epsilon > 0$ and choose
$a\in \cl A^{**}_+$ and $b\in \cl B^{**}_+$
such that $E\leq a\otimes 1 + 1\otimes b$, 
$$E(a\otimes 1) = (a\otimes 1)E \ \mbox{ and } \ 
E(1\otimes b) = (1\otimes b)E,$$
and
$\phi(a) + \psi(b) = \beta(E)$.
Since the elements $a\otimes 1, 1\otimes b$ and $E$ are contained in a common abelian von Neumann algebra, 
by functional calculus, we can assume that 
$\|a\|\leq 1$ and $\|b\|\leq 1$. 
By the Spectral Theorem, there exist  families $(p_i)_{i=1}^n$
(resp. $(q_j)_{j=1}^m$) of mutually orthogonal projections in $\cl A^{**}$ (resp. $\cl B^{**}$) with sum $1$, such that 
$E$ commutes with the family $\{p_i\otimes 1, 1\otimes q_j\}_{i,j}$, 
and scalars $(\lambda_i)_{i=1}^n\in [0,1]$ (resp. $(\mu_j)_{j=1}^m \in [0,1]$) such that, if 
$$a' = \sum_{i=1}^n \lambda_i p_i \ \mbox{ and } \  b' = \sum_{j=1}^m \mu_j q_j,$$
then $a\leq a'$, $b\leq b'$ and
$\phi(a') + \psi(b')  < \beta(E) + \epsilon.$

Set $c = 1 - a'$ and $d = 1 - b'$, $c_i = 1 - \lambda_i$, $d_j = 1 - \mu_j$, $i=1,\ldots,n, j = 1,\ldots,m$. Then
\begin{equation}\label{eq_inverted}
c\otimes 1 + 1\otimes d - 1\otimes 1 \leq E^{\perp}.
\end{equation}
For each $t \in [0,1]$, let $p_t = \sum\{p_i : c_i > t\}$ and
$q_t = \sum\{q_j : d_j > 1 - t\}$.
We claim that 
\begin{equation}\label{eq_ptqt}
p_t\otimes q_t\leq E^{\perp} \ \ \mbox{ for every } t\in [0,1].
\end{equation}
To see this, note that
if $c_i > t$ and $d_j > 1 - t$ then $c_i + d_j - 1 > 0$ and write $F$ for the set of these pairs $(i,j)$ for which these inequalities hold.
By (\ref{eq_inverted}), 
$$\sum_{(i,j)\in F} (c_i + d_j - 1) p_i\otimes q_j \leq E^{\perp}.$$
Now Lemma \ref{l_spand} (ii) implies that $p_i\otimes q_j \leq E^{\perp}$ for every $(i,j) \in F$, and 
(\ref{eq_ptqt}) is proved.

Set $f(t) = \phi(p_t)$ and $g(t) = \psi(q_t)$, $t\in [0,1]$. 
It is straightforward to check that
$$\phi(c) + \psi(d) = \int_0^1 (f(t) + g(t))dt.$$
Since $\phi(c) + \psi(d) > 2 - \beta(E) - \epsilon$,
there exists $t_0\in [0,1]$ such that
$f(t_0) + g(t_0) >  2 - \beta(E) - \epsilon$.
Setting $p = 1 - p_{t_0}$ and $q = 1 - q_{t_0}$, we see that
$E\leq (p\otimes 1) + (1\otimes q)$. Lemma \ref{l_spand} (i) implies that 
$E\leq (p\otimes 1) \vee (1\otimes q)$. Since 
$$\phi(p) + \psi(q) = 2 - f(t_0) - g(t_0) < \beta(E) + \epsilon,$$ 
we have that $\gamma(E)\leq \beta(E)$ and hence, by Theorem \ref{th_bg}, 
$\beta(E) = \gamma(E)$. 
\end{proof}

\noindent {\bf Remarks. (i) } 
As a consequence of Proposition \ref{p_ifcom} and Remark \ref{r_gammadif}, we see that the infimum in the definition of $\beta(E)$, for a projection $E$, is not necessarily achieved on elements $a$, $b$ whose ampliations $a\otimes 1$ and $1\otimes b$ commute with $E$. 

\smallskip


{\bf (ii) } We note that the conclusion of Proposition \ref{p_ifcom} holds true under 
the weaker assumption which does not require that the infimum in the definition of $\beta$ is achieved, 
but that there exists a sequence of pairs $((a_k,b_k))_{k\in \bb{N}}$, 
such that for all $k\in \mathbb{N}$ we have $E(a_k\otimes 1) = (a_k\otimes 1)E$,
$E(1\otimes b_k) = (1\otimes b_k)E$, and 
$\phi(a_k) + \psi(b_k)\to_{k\to \infty} \beta(E)$. 
By functional calculus such pairs exist if $\cl A$ and $\cl B$ are commutative.


\subsection{Monotonicity and preservation} \label{ss_mon}

Let $(\cl A, \phi)$ and $(\cl B, \psi)$ be measured $C^*$-algebras, fixed throughout this subsection.

\begin{proposition}\label{p_monco}
\begin{itemize}
\item[(i)] 
If $(T_k)_{k\in \bb{N}}$ is a sequence in $(\cl A\otimes\cl B)_+$, $T \in(\cl A\otimes\cl B)_+$
and $T_k\to_{k\to\infty} T$ in the weak* topology, then $\alpha(T) \leq \liminf_{k\in \bb{N}} \alpha(T_k)$.
If, in addition, the sequence $(T_k)_{k\in \bb{N}}$ is monotone then 
$\alpha(T) = \lim_{k\in \bb{N}} \alpha(T_k)$.
\item[(ii)] 
The function $\alpha : (\cl A\otimes\cl B)_+ \to \bb{R}^+$ is convex, monotone and 
continuous in the norm topology.

\end{itemize}
\end{proposition}

\begin{proof}
(i)	Assume that $T_k\to_{k\to \infty} T$ in the weak* topology and, using Remark \ref{r_less} (i), let
	$\sigma\in \bist(\phi,\psi)$ have the property $\sigma(T) = \alpha(T)$.
	Then
	$$\alpha(T) = \lim_{k\to\infty} \sigma(T_k) \leq \liminf_{k\in \bb{N}} \alpha(T_k).$$
	
	Now suppose that $T_k\to_{k\to \infty} T$ in the weak* topology and the sequence $(T_k)_{k\in\bb{N}}$
	is monotone. 
	Let $f, f_k : \bist(\phi,\psi) \to \bb{R}^+$ be the functions given by $f(\sigma) = \sigma(T)$ and
	$f_k(\sigma) = \sigma(T_k)$, $k\in \bb{N}$. 
	Then the sequence $(f_k)_{k\in \bb{N}}$ is monotone, consists of continuous functions, 
	and converges to the continuous function $f$. By Dini's Theorem, $f_k\to_{k\to\infty} f$ uniformly; 
	in particular, $\|f_k\|_{\infty} \to_{k\to\infty} \|f\|_{\infty}$, that is, $\alpha(T_k)\to_{k\to\infty} \alpha(T)$.
	
(ii)	It is trivial that, for $S,T\in (\cl A\otimes \cl B)^{**}_+$, the inequality $S\leq T$ implies 
	$\alpha(S)\leq \alpha(T)$. 
	For the convexity, let $S$ and $T$ be positive contractions in $(\cl A\otimes \cl B)^{**}$, and $s,t\in [0,1]$, $s + t = 1$. Then 
\begin{eqnarray*}
\alpha(sS + tT)
& = & 
\sup\{\sigma(sS + tT) : \sigma\in \mathsf{C}(\phi,\psi)\}\\
& \leq & 
\sup\{s\sigma(S) + t\tau(T) : \sigma,\tau\in \mathsf{C}(\phi,\psi)\}
= 
s\alpha(S) + t\alpha(T).
\end{eqnarray*}
	
	Let $(T_k)_{k\in \bb{N}}$ be a sequence of positive elements in $\cl A\otimes\cl B$, and $T$ be a positive element in $\cl A\otimes\cl B$. 
	Assume that $\|T_k - T\|\to_{k\to \infty} 0$ and, using Remark \ref{r_less} (i), let 
	$\sigma_k\in \bist(\phi,\psi)$ be such that 
	$\alpha(T_k) = \sigma_k(T_k)$, $k\in \bb{N}$. 
	Suppose that $\alpha(T_{k_l})\to_{l\to\infty} \delta$ 
	for some subsequence $(k_l)_{l\in \bb{N}}$. 
	By the weak* compactness of $\bist(\phi,\psi)$, we may assume, without loss of generality, that 
	$\sigma_{k_l}\to_{l\to \infty} \sigma$ in the weak* topology, for some $\sigma\in \bist(\phi,\psi)$. 
	For $\epsilon > 0$, let $l_0\in \bb{N}$ be such that 
	$\|T_{k_l} - T_{k_{l_0}}\| < \epsilon$ and $|\sigma(T) - \sigma_{k_l}(T)| < \epsilon$ whenever $l \geq l_0$.
	Then
	\begin{eqnarray*}
		|\sigma(T) - \alpha(T_{k_l})|
		& = & 
		|\sigma(T) - \sigma_{k_l}(T_{k_l})|\\
		&  \leq & 
		|\sigma(T) - \sigma_{k_l}(T)| + |\sigma_{k_l}(T) - \sigma_{k_l}(T_{k_l})|\\
		&  \leq & 
		|\sigma(T) - \sigma_{k_l}(T)| + \|T - T_{k_l}\| < 2\epsilon, 
	\end{eqnarray*}
	whenever $l\geq l_0$.
	It follows that 
	$$\alpha(T) \geq \sigma(T) \geq \alpha(T_{k_l}) - 2\epsilon, \  \ \ \ l\geq l_0,$$
	implying that 
	$\delta \leq \alpha(T).$
	Thus, $\limsup_{k\in \bb{N}} \alpha(T_k) \leq \alpha(T)$. 
	The proof is now complete in view of (i). 
\end{proof}

We next record a simple observation regarding the behaviour of the 
coupling capacity with respect to compositions with maps.
If $(\cl A, \phi)$ and $(\cl B, \psi)$ are unital $C^*$-algebras equipped with states, 
a  positive map $\Theta : \cl A\otimes\cl B \to \cl A\otimes\cl B$ will called $(\phi, \psi)$-reducing if  
$\Theta^*(\bist(\phi,\psi))\subseteq \tilde{\bist}(\phi,\psi)$.

\begin{proposition}\label{p_red}
Let $(\cl A, \phi)$ and $(\cl B, \psi)$ be measured $C^*$-algebras, 
and let $T$ be a positive contraction in $(\cl A\otimes\cl B)^{**}$. 

\begin{itemize}
\item[(i)]
If $\Theta : \cl A\otimes\cl B \to \cl A\otimes\cl B$ is a positive $(\phi, \psi)$-reducing map, then 
$\alpha(\Theta^{**}(T)) \leq \alpha(T)$.

\item[(ii)]
If $\pi \in \textup{Aut}(\cl A)$ and $\rho \in \textup{Aut}(\cl B)$ are automorphisms such that $\phi \circ \pi = \phi$ and $\psi\circ \rho =\psi$, then 
\[ \alpha(T) = \alpha\left((\pi\otimes \rho)^{**}(T)\right).\]
In particular, if $\phi$ and $\psi$ are traces
and $u$ (resp. $v$) is a unitary in $\cl A$
	(resp. $\cl B$), then $\alpha(T) = \alpha\left((u\otimes v)T(u\otimes v)^*\right)$.

\end{itemize}
\end{proposition}

\begin{proof}
(i)
Using Remark \ref{r_less} (ii), we have 
\begin{eqnarray*}
\alpha(\Theta^{**}(T)) 
& = & 
\sup\{\sigma(\Theta^{**}(T)) : \sigma\in \bist(\phi,\psi)\}\\
& \leq & 
\sup\{\sigma'(T) : \sigma'\in \tilde{\bist}(\phi,\psi)\} = \alpha(T).
\end{eqnarray*}

(ii) 
Letting $\Theta = \pi\otimes\rho$, we have that $\Theta$ is invertible, positive, has a positive inverse, 
and $\Theta^*(\bist(\phi,\psi)) = \bist(\phi,\psi)$. The claim therefore follows from (i). 
\end{proof}

\subsection{Dependence on the underlying states}

In the last subsection, the pairs $(\cl A, \phi)$ and $(\cl B, \psi)$ were fixed and 
$\alpha(T)$ was examined as a function on $T$. 
We now briefly change the perspective and look at how $\alpha(T)$ changes if we fix $T$ and allow the states $\phi$ and $\psi$ to vary. 
In order to underline the dependence on the chosen reference states, 
we will write $\alpha_{\phi, \psi}(T)$ (resp. $\beta_{\phi, \psi}(T)$) 
for the parameter $\alpha$ (resp. $\beta$), 
introduced in Definition \ref{def:capac}.
Denote by $S_{\rm f}(\cl A)$ the collection of all faithful states on $S(\cl A)$ (note that $S_{\rm f}(\cl A)$  
is not closed unless $\cl A = \mathbb{C}$).

\begin{proposition}\label{prop:upper}
Fix two unital $C^*$-algebras $\cl A$ and $\cl B$ and a positive contraction  
$T \in \cl A\otimes\cl B$. 
\begin{itemize}
\item[(i)]
The function
\[S(\cl A) \times S(\cl B) \to \bb{R}_+; \ (\phi, \psi) \mapsto \alpha_{\phi, \psi}(T),\]
is upper semicontinuous.

\item[(ii)] If $\cl A$ and $\cl B$ are finite dimensional then 
the function
\[ S_{\rm f}(\cl A) \times S_{\rm f}(\cl B) \to \bb{R}_+; \ (\phi, \psi) \mapsto \alpha_{\phi, \psi}(T),\]
is continuous.
\end{itemize}
\end{proposition}

\begin{proof}
(i) 
Suppose that $(\phi_i, \psi_i)_{i \in I}$ is a net of states, weak$^*$-convergent to 
a pair $(\phi, \psi)\in S(\cl A) \times S(\cl B)$. 
Using Remark \ref{r_less} (i), 
choose $\sigma_i \in \bist(\phi_i, \psi_i)$ such that $\alpha_{\phi_i, \psi_i}(T)=\sigma_i(T)$, for each $i \in I$. 
After passing to a subnet if necessary, we may assume 
that $(\sigma_i)_{i \in I}$ converges to a state $\sigma \in S(\cl A \otimes \cl B)$.
It is clear that $\sigma \in \bist(\phi, \psi)$ and, naturally, 
$\alpha_{\phi, \psi}(T) \geq  \sigma(T) = \lim_{i \in I}\sigma_i(T)$.

(ii) 
Suppose that $(\phi_k, \psi_k)_{k \in \mathbb{N}}$
is a sequence of faithful states, convergent to $(\phi, \psi)\in S_{\rm f}(\cl A) \times S_{\rm f}(\cl B)$. 
For each $k \in \mathbb{N}$, 
choose $a_k \in \cl A_+$ and $b_k \in \cl B_+$  such that $T \leq a_k \otimes I + I \otimes b_k$ and  $\beta_{\phi_k, \psi_k}(T)\geq \phi_k(a_k) + \psi_k(b_k) - \frac{1}{k}$.

We claim that the sequence $(a_k)_{k\in \mathbb{N}}$  is bounded. 
Let $\tau \in S(\cl A)$ be a faithful trace and let $D$ and $D_k$ be (invertible) elements of $\cl A$ such that 
$\phi = \tau(D\cdot)$ and $\phi_k = \tau (D_k \cdot)$, $k\in \bb{N}$. 
Since $D_k\stackrel{k\to \infty}{\longrightarrow} D$, we have that $D_k^{-1}\stackrel{k\to \infty}{\longrightarrow} D^{-1}$. 
In particular, $(D_k^{-1})_{k \in \mathbb{N}}$ is bounded. By finite dimensionality, it follows that
\[ \|a_k\| \leq M \|D_k a_k\| \leq MC \tau(D_k a_k) = MC \phi_k(a_k), \ \ \ k \in \mathbb{N},\]
for some positive constants $M$ and $C$, 
depending only on $\cl{A}$ and the sequence $(\phi_k)_{k\in \mathbb{N}}$.
Since the sequence $(\phi_k(a_k))_{k\in \mathbb{N}}$ is bounded, so is 
the sequence $(a_k)_{k\in \mathbb{N}}$  is bounded; by symmetry, so is 
the sequence $(b_k)_{k\in \mathbb{N}}$.

After passing to subsequences if necessary, $a_k\to_{k\to \infty} a$ and 
$b_k\to_{k\to \infty} b$ for some $a \in \cl A_+$ and $b \in \cl B_+$. 
We have $T \leq a \otimes I + I \otimes b$ and  
$$\beta_{\phi, \psi}(T)\leq \phi(a) + \psi(a) = \lim_{k \to \infty} \phi_k(a_k) + \psi_k(b_k) - \frac{1}{k} \leq \beta_{\phi_k, \psi_k}(T).$$
The claim now follows after an application of Theorem \ref{th_bg} and Proposition \ref{prop:upper} (i).
\end{proof}


\subsection{The commutative case}\label{s_comcase}

In this section, we assume that $\cl A$ and $\cl B$ are abelian. We will see that, in this case, the coupling capacity
$\alpha$ coincides with some 
previously studied 
parameters, appearing before in the theory of optimal transport and in operator algebra theory. 

We first note that, by the Gelfand Theorem, 
every unital abelian C*-algebra is *-isomorphic to 
the C*-algebra $C(X)$, for some compact Hausdorff space $X$. 
If $X$ is a compact Hausdorff space, 
we write $\cl F_X$ for the $\sigma$-algebra of Borel subsets of $X$. 
Given $\alpha\in \cl F_X$, 
the linear functional $e_{\alpha} : M(X)\to \bb{C}$, given by
$$e_{\alpha}(\mu) = \mu(\alpha), \ \ \ \mu\in M(X),$$
is bounded with $\|e_{\alpha}\| = 1$, and hence gives rise to
an element of $C(X)^{**}$, which will be denoted in the same way.
By abuse of notation, we identify $e_{\alpha}$ with the characteristic function $\chi_{\alpha}$ of $\alpha$, 
thus viewing $\chi_{\alpha}$ as an element of $C(X)^{**}$.

We fix compact Hausdorff spaces $X$ and $Y$, and set $\cl A = C(X)$ and $\cl B = C(Y)$. 
Fix Borel probability measures $\mu$ and $\nu$ on $X$ and $Y$, respectively. 
We will write $L^p(X)$ and $L^p(Y)$ for the corresponding $L^p$-spaces, where $p \in \{1,\infty\}$, with 
respect to $\mu$ and $\nu$, respectively. 
We equip $X\times Y$ with the product $\sigma$-algebra $\cl F_{X,Y}$, that is the $\sigma$-algebra 
generated by the sets $A\times B$, where $A\in \cl F_X$ and $B\in \cl F_Y$; 
note that $\cl F_{X,Y}$ is contained in the Borel $\sigma$-algebra $\cl F_{X\times Y}$ of
$X\times Y$. 
Given a positive measure $\sigma$ on $(X\times Y,\cl F_{X,Y})$, let $\sigma^*$ be the outer measure associated with $\sigma$ and let 
$\sigma_X$ (resp. $\sigma_Y$) be the $X$-marginal (resp. the $Y$-marginal) of $\sigma$.

Let $\kappa \subseteq X\times Y$. The following parameters, associated with $\kappa$, were defined in \cite{hs}:

\begin{itemize}
\item[(i)] 
$\alpha(\kappa) = \sup\{\sigma^*(\kappa) : \sigma_X\leq \mu, \sigma_Y\leq \nu\}$;

\item[(ii)]
$\beta(\kappa) = \inf\{\int_X a d\mu+ \int_Y b d\nu : 
a\in L^{\infty}(X), b\in L^{\infty}(Y), a(x) + b(y)\geq 1 \mbox{ on } \kappa\}$;

\item[(iii)]
$\gamma(\kappa) = \inf\{\mu(A) + \nu(B) : A\in \cl F_X, B\in \cl F_Y, \kappa \subseteq (A\times Y)\cup (X\times B)\}$.
\end{itemize}

We will now show that the above parameters coincide with these studied in our paper.

\begin{proposition}\label{l_mc}
Let $(X,\mu)$ and $(Y,\nu)$ be probability spaces and $\kappa\in \cl F_{X,Y}$. 
Then 
$\alpha(\kappa) = \alpha(\chi_\kappa)$, $\beta(\kappa) = \beta(\chi_\kappa)$ and 
$\gamma(\kappa) = \gamma(\chi_\kappa)$.
\end{proposition}

\begin{proof}
Since $\kappa\in \cl F_{X,Y}$, we have that $\sigma^*(\kappa) = \sigma(\kappa)$, and hence the claim about the parameter $\alpha$ follows from Remark \ref{r_less} (ii). 

Moving to $\beta$, let $\pi_{\mu} : C(X)\to \cl B(L^2(X,\mu))$ be the *-representation given by $\pi_{\mu}(a)\xi = a\xi$, $a\in C(X)$, $\xi\in L^2(X,\mu)$. 
Extend $\pi_{\mu}$ to a normal *-representation (denoted in the same way) 
$\pi_{\mu} : C(X)^{**}\to \cl B(L^2(X,\mu))$; it is clear that its range can be canonically identified with $L^{\infty}(X,\mu)$ and we hence obtain a *-epimorphism $\pi_{\mu} : C(X)^{**}\to L^{\infty}(X,\mu)$. 
Similarly, we have a *-epimorphism $\pi_{\nu} : C(Y)^{**}\to L^{\infty}(Y,\nu)$.

Note that, given 
$a\in L^{\infty}(X,\mu)$ and $\tilde{a}\in C(X)^{**}$ such that $\pi_{\mu}(\tilde{a}) = a$, 
(resp. $b\in L^{\infty}(Y,\nu)$ and $\tilde{b}\in C(Y)^{**}$ such that $\pi_{\nu}(\tilde{b}) = b$),
we have 
$$\langle \tilde{a},\mu\rangle = \int_{X} ad\mu
\ \ \mbox{(resp. } \langle \tilde{b},\nu\rangle = \int_{Y} bd\nu\mbox{)}.$$ 
Assume that $a(x) + b(y) \geq 1$ on $\kappa$. This 
means that 
$$(\pi_{\mu}\otimes \pi_{\nu})
(\tilde{a}\otimes 1 + 1 \otimes \tilde{b} - 
\chi_\kappa) \geq 0.$$
Using the fact that *-epimorphisms are complete quotient maps, we conclude that 
$\tilde{a}\otimes 1 + 1 \otimes \tilde{b} - 
\chi_\kappa \geq 0$, at the expense of possibly changing $\tilde{a}$ and $\tilde{b}$, while retaining their positivity and the properties
$\pi_{\mu}(\tilde{a}) = a$ and $\pi_{\nu}(\tilde{b}) = b$. 
These arguments show that $\beta(\kappa) = \beta(\chi_\kappa)$. 
Finally, the claim about the parameter $\gamma$ are obtained from the one about $\beta$ after restricting $a$ and $b$ to be projections. 
\end{proof}

\begin{remark}\label{r_consequ}
\rm 
By Proposition \ref{l_mc}, as consequences of 
Theorem \ref{th_bg} and Proposition \ref{p_ifcom} (together with the remarks following the latter)
we obtain the fact that, whenever $\kappa\subseteq X\times Y$ is a clopen set, we have that 
$\alpha(\kappa) = \beta(\kappa) = \gamma(\kappa)$. 
The latter equalities are very special instances of Corollaries of Lemma 1 and Theorem 1 in \cite{hs} which, in their turn, are quantitative versions of 
Arveson's Null Set Theorem \cite[Section 1.4]{a}. Naturally the results of \cite{hs} apply in much greater generality; we will see a special instance of this below.


\end{remark}

\begin{proposition} \label{prop:lowerscts}
Let $(\cl A, \phi)$ and $(\cl B, \psi)$ be measured abelian $C^*$-algebras and suppose that $T\in (\cl A \otimes \cl B)_+^{**}$ is lower semicontinuous, i.e.\ there exists an increasing net $(T_i)_{i \in I}$ with $T_i \in (\cl A \otimes \cl B)_+$ which converges to $T$ in weak*-topology. Then $\alpha(T) = \beta(T)$.
\end{proposition}
\begin{proof}
Fix $\epsilon > 0$. Note first that by functional calculus  for each $i 
\in I$ we have
$$\beta(T_i) = \inf\{\phi(a) + \psi(b) : a\in \cl A^{**}, b\in \cl B^{**}, 
\|a\|, \|b\|\leq \|T\|, T_i\leq a\otimes 1 + 1 \otimes b\}.$$
For each $i \in I$ let then $a_i\in \cl A^{**}$ and $b_i\in \cl B^{**}$ be such that 
$\|a_i\|\leq \|T\|$, $\|b_i\|\leq \|T\|$, $T_i\leq a_i\otimes 1+1\otimes b_i$, and 
$$\phi(a_i) + \psi(b_i) \leq \beta(T_i) + \epsilon.$$
By passing to a subnet if necessary, assume that 
$$a_i\to_{i\in I} a \ \mbox{ and } \ b_i\to_{i\in I} b$$
in the weak* topologies of $\cl A^{**}$ and $\cl B^{**}$, respectively. 
We have that $T\leq a\otimes 1 + 1 \otimes b$.

Since $T_i \in \cl A\otimes\cl B$, by Theorem \ref{th_bg} we have $\alpha(T_i) = \beta(T_i)$, $i\in I$. 
There exists $i_0\in I$ such that, if $i\geq i_0$ then 
$$\beta(T) \leq \phi(a) + \psi(b)  \leq \phi(a_i) + \psi(b_i) + \epsilon 
\leq \beta(T_i) + 2\epsilon = \alpha(T_i) + 2\epsilon
\leq \alpha(T) + 2\epsilon,$$
where we have used the monotonicity of $\alpha$ for the last inequality. 
We conclude that $\beta(T)\leq \alpha(T)$, and the converse was already noted in Theorem \ref{th_bg}. 	
	
\end{proof}

\begin{remark}
\rm
Let $c : X\times Y\to [0,1]$ be a lower semi-continuous function. Then $c$ can be viewed as an element of 
$C(X\times Y)^{**}$ in a natural fashion (this was detailed in the second paragraph of this section in the case of 
characteristic functions of Borel sets). 
We can rewrite the equality between the parameters $\alpha$ and 
$\beta$ from the proposition above as the equality 
\begin{eqnarray*}
	& & 
	\hspace{-0.6cm} \sup \hspace{-0.05cm}\left\{\int_{X\times Y} \hspace{-0.1cm} c d\sigma : \sigma_X = \mu, \sigma_Y = \nu\right\} = \\
	& & 
	\hspace{-0.6cm} 
	\inf\hspace{-0.05cm}\left\{\hspace{-0.05cm}\int_X a d\mu \hspace{-0.05cm} + \hspace{-0.2cm} \int_Y b d\nu : 
	a\hspace{-0.05cm} \in \hspace{-0.05cm} L^{\infty}\hspace{-0.05cm}(X), 
	b \hspace{-0.05cm}\in\hspace{-0.05cm} L^{\infty}\hspace{-0.05cm}(Y), 
	c(x,y)\hspace{-0.05cm}\leq\hspace{-0.05cm} a(x) \hspace{-0.05cm} + \hspace{-0.05cm} 
	b(y) \mbox{ on } X\hspace{-0.07cm} \times\hspace{-0.07cm} Y\hspace{-0.05cm}\right\}\hspace{-0.1cm}.
\end{eqnarray*}
In the case under consideration, $L^{\infty}(X)\subseteq L^1(X)$ and $L^{\infty}(Y)\subseteq L^1(Y)$. 
It follows that the displayed equality persists if the infimum is taken after replacing $L^{\infty}(X)$ (resp. $L^{\infty}(Y)$)
by $L^1(X)$ (resp. $L^1(Y)$). Thus in this special case we recover 
the well-known Monge-Kantorovich duality formula in the theory of optimal transport 
(see e.g. \cite[Theorem 1.3]{villani0}). 
	
\end{remark}


\section{The matrix case}\label{s_mc}

In this section, we consider the simplest non-commutative case,
where 
$\cl A = \cl L(\bb{C}^n) \equiv M_n$, $\cl B = \cl L(\bb{C}^m) \equiv M_m$, for some fixed $n,m\in\bb{N}$.
We first show that the quantum Strassen theorem proved in \cite{zyyy} 
can be obtained as a consequence of Theorem \ref{th_bg}. 
For a subspace $\cl X\subseteq \mathbb C^n\otimes\mathbb C^m$ write $E_{\cl X}$ for the projection onto $\cl X$. For $\sigma\in (M_n\otimes M_m)_+$ write 
$$\supp\sigma=\{\xi\in \mathbb C^n\otimes\mathbb C^m: \langle\sigma\xi,\xi\rangle=0 \}^\perp.$$ 
In the sequel, it will be convenient to write $M_n^+$ and $M_n^h$ instead of $(M_n)_+$ and $(M_n)_h$, respectively. Recall that if $\phi$ is a state on $M_n$ we denote its associated density matrix by $A_\phi$.

\begin{proposition}
[Quantum Strassen Theorem \cite{zyyy}]
\label{qstrassen} 
Let $\cl X$ be a subspace of 
$\mathbb C^n\otimes \mathbb C^m$, 
$\phi$ (resp. $\psi$) be a state on $M_n$ (resp. $M_m$) and 
$\rho_1\in M_n^+$ (resp. $\rho_2\in M_m^+$) be such that $A_{\phi} = \rho_1$ (resp. $A_{\psi} = \rho_2$).
The following are equivalent:
\begin{itemize}
\item[(i)] 
$\alpha(E_{\cl X})=1$;

\item[(ii)] 
there is a coupling $\sigma\in \bist(\phi,\psi)$ such that $\supp \sigma\subseteq \cl X$;

\item[(iii)] $\tr(\rho_1a_1)\leq\tr(\rho_2a_2)$ whenever $a_1\in M_n^h$, $a_2\in M_m^h$ are 
such that $E_{\cl X^\perp}\geq a_1\otimes I_m-I_n\otimes a_2$.
\end{itemize}
    \end{proposition}

\begin{proof}
 (i)$\Leftrightarrow$(ii) It is enough to note that if $\sigma$ is a state on $M_n\otimes M_m$ 
then $\supp\sigma\subseteq\cl X$ if and only if $\sigma(E_{\cl X})=1$. In fact, 
\begin{eqnarray*}
& & 
\sigma (E_{\cl X})=1 \Longleftrightarrow \sigma(I - E_{\cl X}) = 0\\
& &
\Longleftrightarrow \sigma(\xi\xi^*) = \tr(\sigma\xi\xi^*) = \frac{1}{nm}\langle\sigma\xi,\xi\rangle = 0\  \mbox{ for all } \xi\in{\cl X}^\perp\\
& & 
\Longleftrightarrow \supp\sigma\subseteq \cl X.
 \end{eqnarray*}
 
 (i)$\Leftrightarrow$(iii) By Theorem \ref{th_bg}, $\alpha(E_{\cl X})=\beta(E_{\cl X})$. The fact that $\beta(E_{\cl X})=1$ is equivalent to 
 \begin{equation}\label{impl}
 E_{\cl X}\leq a\otimes I_m+I_n\otimes b\Rightarrow \phi(a)+\psi(b)\geq 1
 \end{equation}
 whenever $a\in M_n^+$, $b\in M_m^+$ and by the arguments in the proof of Theorem \ref{th_bg}, 
 whenever $a$, $b$ are hermitian.
 Letting $a_1=1-a$, $a_2=b$, (\ref{impl}) can be rewritten as 
 $$E_{\cl X^\perp}\geq a_1\otimes I_m-I_n\otimes a_2\ \Longrightarrow \ \phi(a_1)\leq \psi(a_2),$$
 giving the desired equivalence. 
\end{proof}

In view of Proposition \ref{qstrassen}, we see that, in the case of matrix algebras, 
Theorem \ref{th_bg} can be viewed as a quantitative and non-commutative 
extension of the Quantum Strassen Theorem.

\begin{remark}
\rm 
We note that the equivalence (i)$\Leftrightarrow$(ii) in Proposition \ref{qstrassen} persists in the general case
of measured C*-algebras $(\cl B(H_1),\phi)$ and $(\cl B(H_2),\psi)$, with $H_1, H_2$ Hilbert spaces (possibly infinite dimensional), $\phi$ and $\psi$ normal states,  and the subspace $\cl X$ replaced by an 
arbitrary projection $E\in \cl B(H_1 \otimes H_2)$. Together with a straightforward approximation argument it can be used to infer \cite[Theorem 4.3]{fgz}. 
\end{remark}

In the rest of the section both algebras $M_n$ and $M_m$ will be equipped with normalised traces $\tr$. 
As customary, we abbreviate \lq\lq completely positive and trace preserving'' to \lq\lq cptp'', 
and note that trace preservation is with respect to the normalised traces.

Recall that, given a map $\Phi:M_n \to M_m$, its associated \emph{Choi matrix} 
$\Gamma_\Phi \in M_n \otimes M_m$ is given by letting
\begin{equation}\label{eq_Phi}
(\Gamma_\Phi)_{i,j} = \Phi(\epsilon_{i,j}), \;\;\;i, j = 1,\ldots, n.
\end{equation}
Conversely, each matrix $\Gamma \in M_n (M_m)$ determines, via (\ref{eq_Phi}), 
a linear map $\Phi_\Gamma : M_n \to M_m$. 
The next statement, which characterises the elements of the set $\bist(\tr_n,\tr_m)$, 
is rather well-known and for $m=n$ is precisely \cite[Theorem 2.2]{ohno}. We include a straightforward proof for the convenience of the reader.

\begin{proposition}\label{lemma_choi}
Let $\sigma\in (M_n\otimes M_m)^*$. Recall that $A_\sigma \in M_n\otimes M_m=M_n(M_m)$ denotes the density matrix of $\sigma$. The following are equivalent: 
\begin{itemize}

\item[(i)] 
$\sigma\in\bist(\tr_n,\tr_m)$; 

\item[(ii)] 
$\frac{1}{n}\Phi_{A_{\sigma}}$ is unital and trace preserving.   
\end{itemize}
\end{proposition}

\begin{proof}
(i)$\Rightarrow$(ii) 
To lighten notation, we set $\Phi = \Phi_{A_{\sigma}}$. 
Let $A_\sigma = (B_{i,j})_{i,j=1}^n\in M_n\otimes M_m$ 
(so that we have $\Phi(\epsilon_{i,j}) = B_{i,j}$ for each $i,j=1,\ldots,n$).
For $b \in M_m$ we have 
$$\tr( A_\sigma (I\otimes b))=\frac{1}{n}\sum_{i=1}^n\tr\nolimits_m{(B_{i,i}b)} = \tr\nolimits_m{\left(\frac{1}{n}\left(\sum_{i=1}^n B_{i,i}\right)b\right)}=\tr\nolimits_m{(b)},$$ so that
    $\frac{1}{n}\sum_{i=1}^nB_{i,i} = I_m$. 
    Therefore,
    $$\frac{1}{n}\Phi(I_n) = \frac{1}{n}\sum_{i=1}^n\Phi(\epsilon_{i,i}) = \frac{1}{n}\sum_{i=1}^nB_{i,i} = I_m.$$
Further, for $a = (a_{i,j})_{i,j=1}^n\in M_n$, we have 
    $$\tr(A_\sigma(a\otimes I))=\frac{1}{n}\sum_{i,j=1}^n\tr\nolimits_m{(B_{i,j}a_{j,i})} = \tr\nolimits_n{(a)}.$$
    Taking $a=\epsilon_{l,k}$ for $k,l=1,\ldots,n$ we obtain $\tr_m(\Phi(\epsilon_{k,l})) = \tr_m{(B_{k,l})}=\delta_{k,l} = n \tr_n (\epsilon_{k,l})$, which implies that $\frac{1}{n}\Phi$ is trace-preserving.  
    
(ii)$\Rightarrow$(i) follows by reversing the arguments in the previous paragraph. 
\end{proof}


Given a vector $\xi \in \cnm$, we write $S_{\xi}$ for the linear transformation 
from $\bb{C}^n$ into $\bb{C}^m$ corresponding to $\xi$
in the canonical way, so that $S_{e\otimes f} = fe^*$, $e\in \bb{C}^n$, $f\in \bb{C}^m$.  The singular value decomposition of $S_\xi$ 
allows us to find  (assuming, say, that $n\leq m$) a descending sequence of scalars $\lambda_1 \geq \lambda_2\geq \dots \geq \lambda_n\geq 0$ and  orthonormal collections $(e_i)_{i=1}^n\subseteq \bb{C}^n$ and $(f_i)_{i=1}^n\subseteq \bb{C}^m$ such that 
$\xi = \sum_{i=1}^n \lambda_i e_i \otimes f_i$. We will call any such decomposition a \emph{Schmidt decomposition} for $\xi$. Note that while the decomposition itself is not unique, the scalars $\lambda_i$ are determined uniquely.

Let $\xi\in \cnm$ be a unit vector and set $E_{\xi} = \xi\xi^*$.
The vector $\xi\in \cnm$ is often identified with the pure state with density matrix $E_{\xi}$. 
Under this identification, $\xi$ is called a \emph{separable state}, if 
$\xi = e\otimes f$ for some unit vectors $e\in \bb{C}^n$ and $f\in \bb{C}^m$. 
If $\xi$ is not separable, it is called an \emph{entangled state}; $\xi$ is further called 
\emph{maximally entangled} if (assuming $n\leq m$) there exist 
orthonormal sequences $(e_i)_{i=1}^n$ and $(f_i)_{i=1}^n$ in $\bb{C}^n$ and $\bb{C}^m$, 
respectively, such that $\xi = \frac{1}{\sqrt{n}} \sum_{i=1}^n e_i\otimes f_i$. Note that each of the conditions above has a simple description in terms of the Schmidt decomposition of $\xi$.

We first note an equivalent expression for $\alpha$ that will be  useful later. 

\begin{proposition}\label{p_eeal}
	Let $T\in M_n\otimes M_m$ be a positive contraction and let $\zeta$ be a maximally entangled vector in $\bb{C}^n\otimes \bb{C}^n$ of the form $\zeta = \frac{1}{\sqrt{n}}\sum_{i=1}^n e_i\otimes e_i$, where $\{e_i\}_{i=1}^n$ is an orthonormal basis of $\bb{C}^n$.
	We have that 
$$	\alpha(T) = \max\left\{\langle \Phi^{(n)}(T)\zeta,\zeta\rangle \ : \ 
\Phi: M_m\to M_n \mbox{ is a unital cptp map}\right\}.$$ 
\end{proposition}

\begin{proof}
	Write $T = (T_{i,j})_{i,j=1}^n$, $T_{i,j}\in M_m$. 
	As in the proof of Proposition \ref{lemma_choi}, 
	for $\sigma \in \bist(\tr_n,\tr_m)$, set $\Phi = \Phi_{A_{\sigma}}$; thus, 
	$\frac{1}{n}\Phi : M_n\to M_m$ is a unital quantum channel. 
	Write, further, $A_\sigma = (\sigma_{i,j})_{i,j=1}^n$, where $\sigma_{i,j}\in M_m$. 
	We have 
	\begin{eqnarray*}
		\tr(\sigma T) 
		& = & \frac{1}{n}\sum_{i,j=1}^n \tr\nolimits_m \left(\sigma_{j,i}T_{i,j}\right) 
		= 
		\frac{1}{n}\sum_{i,j=1}^n \tr\nolimits_m \left(\Phi(\epsilon_{j,i})T_{i,j}\right)\\
		& = & 
		\frac{1}{n}\sum_{i,j=1}^n \tr\nolimits_n \left(\epsilon_{j,i} \Phi^*(T_{i,j})\right)
		= 
		\tr\left((\epsilon_{i,j})_{i,j=1}^n\Phi^{* (n)}(T)\right)\\
		& = & 
		\tr\left(\Phi^{* (n)}(T)\cdot n\zeta\zeta^*\right)
		= 
		\frac{1}{n}\left\langle\Phi^{* (n)}(T)\zeta,\zeta\right\rangle.
	\end{eqnarray*}
	The claim follows now  by noting that a map $\Psi : M_n\to M_m$ is unital and trace preserving if and only if so is its dual.
\end{proof}

\begin{proposition}\label{p_sep_alpha}
	Let $\xi$ be a unit vector in $\mathbb C^n\otimes\mathbb C^n$. 
	Write $\xi=\sum_{i=1}^n\lambda_i e_i\otimes f_i$ for its Schmidt decomposition. Then
	\begin{equation}\label{eq_eifi}
	\alpha(E_\xi)\geq \frac{1}{n}\left(\sum_{i=1}^n\lambda_i\right)^2 \geq \frac{1}{n}.
	\end{equation}
	Moreover, for $n=2$ the first inequality is an equality.
	\end{proposition}
	
	\begin{proof}
	Set $\zeta = \frac{1}{\sqrt{n}}\sum_{i=1}^n e_i\otimes e_i$.
	By convexity, the expression for $\alpha(E_\xi)$ in Proposition \ref{p_eeal}
	can be restricted to 
	the extreme points in the (convex) set of all unital quantum channels $\Phi$. 
	If $\Phi_U(T)=UTU^*$, where $U$ is unitary, then $\Phi_U$ is an extreme unital quantum channel. 
	 We have 
	 \begin{eqnarray*}
	 \left\langle\Phi^{(n)}(E_\xi)\zeta,\zeta\right\rangle
	 & = &
	 \left\langle(I\otimes U)E_\xi(1\otimes U)^*\zeta,\zeta\right\rangle
	 = \left\langle E_{(I\otimes U)\xi}\zeta,\zeta\right\rangle\\
	 & = &
	 \left|\langle (I\otimes U)\xi,\zeta\rangle\right|^2
	 = \frac{1}{n}\left|\sum_{i=1}^n\lambda_i\langle Uf_i,e_i\rangle\right|^2
	 \leq \frac{1}{n}\left(\sum_{i=1}^n\lambda_i\right)^2\hspace{-0.1cm}.
	 \end{eqnarray*}
If $U$ is the unitary, given by $Uf_i=e_i$, $i=1,\ldots, n$, then 
$\langle\Phi^{(n)}(E_\xi)\zeta,\zeta\rangle=\frac{1}{n}\left(\sum_{i=1}^n\lambda_i\right)^2$,  
and the first inequality in (\ref{eq_eifi}) follows. 
On the other hand, 
$$1=\|\xi\|^2=\sum_{i=1}^n\lambda_i^2\leq\left(\sum_{i=1}^n\lambda_i\right)^2,$$ 
which implies the second inequality in (\ref{eq_eifi}). 
	 
If $n = 2$ then the channels of unitary conjugation 
exhaust the extreme points of the convex set of all unital quantum channels 
\cite{kummer, bps}, and the claim follows from the previous paragraph. 
\end{proof}

Let 
$$w(\xi)
\hspace{-0.05cm}=\hspace{-0.05cm}
\inf\hspace{-0.05cm}\left\{\tr(a)\hspace{-0.05cm}+\hspace{-0.05cm}\tr(b) : a\in M_n^+, b\in M_m^+, 
E_{\xi}\leq E_{\xi}((a\otimes 1)\hspace{-0.08cm}+\hspace{-0.08cm} (1\otimes b))E_{\xi}\right\}\hspace{-0.05cm}.$$
Clearly, 
\begin{equation}\label{wleqgamma}
w(\xi) \leq \beta (E_\xi),
\ \ \ \xi\in \bb{C}^n\otimes\bb{C}^m, \|\xi\| = 1.
\end{equation}

Let $\Tr_A : M_n\otimes M_m\to M_m$ be the partial trace map, defined by the identity 
$$\Tr(\Tr\hspace{-0.06cm}\mbox{}_A(T)B) = \Tr(T(I\otimes B)), \ \ \ B\in M_m, \ T\in M_n\otimes M_m.$$
The partial trace $\Tr_B : M_n\otimes M_m\to M_n$ is defined similarly.

\begin{lemma}\label{p_w}
	Let $\xi \in \cnm$ be a unit vector and let $m\geq n$.  Then 
	$$w(\xi) = \frac{1}{m\|\Tr\mbox{}_B(E_{\xi})\|} = \frac{1}{m\|\Tr\mbox{}_A(E_{\xi})\|}.$$
	In particular, $w(\xi) \geq \frac{1}{m}$.
\end{lemma}

\begin{proof}
Fix a Schmidt decomposition
	$\xi = \sum_{i=1}^n \lambda_i e_i \otimes f_i$.
	A direct verification shows that 
	$$\Tr\hspace{-0.06cm}\mbox{}_B(E_{\xi}) = \sum_{i=1}^n \lambda_i^2 e_i e_i^*,$$
	and hence $\|\Tr_B(E_{\xi})\| = \lambda_1^2 = \|\Tr_A(E_{\xi})\|$. 

Note that if $a\in M_n^+$ and $b\in M_m^+$ then 
$$E_{\xi}\leq E_{\xi}((a\otimes 1) + (1\otimes b))E_{\xi} \ \Longleftrightarrow \ 
\left\langle ((a\otimes 1) + (1\otimes b))\xi,\xi \right\rangle \geq 1,$$
and the latter inequality can be rewritten as 
\begin{align*} 1 &\leq \sum_{i,j=1}^n \lambda_i \lambda_j \langle (a\otimes 1+1\otimes b)e_i \otimes f_i, e_j \otimes f_j \rangle 
= \sum_{i=1}^n \lambda_i^2 (\langle a e_i, e_i\rangle + \langle b f_i, f_i \rangle).
\end{align*}
In evaluating $w(\xi)$, we are thus led to minimising the expression
$\frac{1}{n}\sum_{i=1}^n \mu_i + \frac{1}{m}\sum_{j=1}^m \nu_j$ over all non-negative scalars $\mu_1, \ldots, \mu_n, \nu_1, \ldots, \nu_m$, satisfying the relation 
$\sum_{i = 1}^n \lambda_{i}^2 \left(\mu_i + \nu_i\right) \geq 1 $. 
Setting $\mu_{n+1}=\ldots=\mu_m = 0$, we have
$$\frac{1}{n}\sum_{i=1}^n \mu_i + \frac{1}{m}\sum_{j=1}^m \nu_j\geq \frac{1}{m}\sum_{j=1}^m (\mu_j+\nu_j),$$ and \begin{eqnarray*}
&&\min\left\{\frac{1}{m}\sum_{i=1}^m (\mu_i + \nu_i) : 
		\sum_{i = 1}^n \lambda_{i}^2 \left(\mu_i + \nu_i\right) \geq 1\right\}\\
		& = & 
		\min\left\{\frac{1}{m}\sum_{i=1}^n (\mu_i + \nu_i) : 
		\sum_{i = 1}^n \lambda_{i}^2 \left(\mu_i + \nu_i\right) \geq 1\right\}\\
		&=&\min\left\{\frac{1}{m}\sum_{i=1}^n \nu_i : 
		\sum_{i = 1}^n \lambda_{i}^2 \nu_i \geq 1\right\} = \frac{1}{m\lambda_1^2}.
	\end{eqnarray*}
It follows that $w(\xi) \geq\frac{1}{m\lambda_1^2}$. 
On the other hand, by taking 
$\nu_1=\frac{1}{\lambda_1^2}$ and  $\nu_i = 0$ for $i > 1$, 
we have that $\sum_{i = 1}^n \lambda_{i}^2 \left(\mu_i + \nu_i\right) \geq 1$ 
and $\frac{1}{n}\sum_{i=1}^n \mu_i + \frac{1}{m}\sum_{j=1}^m \nu_j=\frac{1}{m\lambda_1^2}$, giving $w(\xi)=\frac{1}{m\lambda_1^2}$.
\end{proof}

\begin{theorem}\label{p_sep}
	Let $\xi$ be a unit vector in $\cnm$, and assume that $n\leq m$.  
	Then
	\begin{itemize}
	\item[(i)] $\xi$ is separable if and only if $\alpha(E_{\xi}) = \frac{1}{m}$, if and only if $\gamma(E_{\xi}) = \frac{1}{m}$;
	\item[(ii)] $\xi$ is maximally entangled if and only if $\alpha(E_{\xi}) = 1$.
	\end{itemize}
\end{theorem}

\begin{proof}
(i) Let $\pi\in \{\alpha,\gamma\}$. 
	Suppose first that $\xi$ is separable, that is, $\xi = e\otimes f$ for some unit vectors $e\in \bb{C}^n $ and $f\in \bb{C}^m$. 
	We have that $E_{\xi} \leq 1\otimes (ff^*)$ and hence, 
	by the monotonicity of $\gamma$, we have that
	$\gamma(E_{\xi})  \leq \tr_m (ff^*) = \frac{1}{m}$. 
	It follows from Theorem \ref{th_bg}, inequality (\ref{wleqgamma}) and Lemma \ref{p_w},  that $\pi(E_{\xi}) = \frac{1}{m}$.
	
	Suppose that $\pi(E_{\xi}) = \frac{1}{m}$ for some $\pi\in \{\alpha,\gamma\}$.
	By Theorem \ref{th_bg}, inequality \eqref{wleqgamma} and Lemma \ref{p_w}, $w(\xi) = \frac{1}{m}$. 
	By Lemma \ref{p_w} again, $\|\Tr_B(E_{\xi})\| = 1$. Thus, $S_{\xi}$ has rank one;
	equivalently, $\xi$ is separable.
	
(ii)	Suppose that $\xi$ is maximally entangled. 
	Then, by Proposition \ref{p_sep_alpha}, $\alpha(E_{\xi}) \geq 1$. 
	By Theorem \ref{th_bg}, $\alpha(E_{\xi}) = 1$. 
	
	Conversely, suppose that $\alpha(E_{\xi}) = 1$. By Proposition \ref{qstrassen}, there exists a state
	$\sigma\in \bist(\tr_n,\tr_m)$ supported in the one-dimensional space generated by  $\xi$. 
	Thus $A_\sigma$ is a multiple of $\xi\xi^*$. Since $\tr(A_\sigma)=1$ and $\tr(\xi\xi^*)=\frac{1}{nm}$, 
	we have that $A_\sigma=(nm)\xi\xi^*$. Write $\xi = \sum_{i=1}^n e_i\otimes \xi_i$, where $(e_i)_{i=1}^n$ is the canonical basis of $\bb{C}^n$
	and $\xi_1,\ldots, \xi_n \in \bb{C}^m$. We have $A_\sigma=(nm)(\xi_i\xi_j^*)_{i,j=1}^n$. The condition $\sigma\in \bist(\tr_n,\tr_m)$ implies that for each $i,j=1, \ldots, n$, we have 
	$$ \langle\xi_i,\xi_j\rangle= \frac{nm}{n}\tr\nolimits_{m}(\xi_i\xi_j^*)
	= \tr\left(A_\sigma(\epsilon_{i,j}\otimes I)\right) = \tr\nolimits_n(\epsilon_{i,j}) = \frac{1}{n}\delta_{i,j},$$
	giving that $\xi$ is maximally entangled. 
\end{proof}

\begin{corollary}\label{c_cont}
	The set of values of $\alpha$ on non-zero projections in $M_n\otimes M_n$ is $[\frac{1}{n},1]$. 
	Moreover, if $E\in M_n\otimes M_n$ is a projection, then $\alpha(E)=\frac{1}{n}$ if and only if either $E=\tilde E\otimes ee^*$ or $E=ee^*\otimes \tilde E$ for a projection $\tilde E\in M_n$ and a unit vector $e\in \mathbb C^n$.
\end{corollary}

\begin{proof}
Let $t\to \eta_t$ be a continuous function from $[0,1]$ into $\bb{C}^n\otimes\bb{C}^n$ such that 
	$\eta_0$ is separable, while $\eta_1$ is maximally entangled. Note that the corresponding function $t \mapsto E_{\eta_t}$ is norm continuous.  
	By Theorem \ref{p_sep} and Proposition \ref{p_monco} (ii), the set 
	$\{\alpha(E_{\eta_t}) : t\in [0,1]\}$ coincides with the interval $[\frac{1}{n},1]$. 
	
	Now let $E$ be a projection in $M_n$ and assume that $\alpha(E)=\frac{1}{n}$. 
	By monotonicity, $\alpha(E)\geq\alpha(E_\xi)\geq\frac{1}{n}$ for any  unit vector $\xi$ in the range of $E$; 
	using Theorem \ref{p_sep}, we obtain that any vector in the range of $E$ is separable from which 
	easily implies (arguing by contradiction)
	that $E$ is either $\tilde E\otimes ee^*$ or $ee^*\otimes\tilde E$ for some projection $\tilde E\in M_n$ and some 
	unit vector $e\in\mathbb C^n$. The converse implication follows from Proposition \ref{p_papb}. 
\end{proof}

\begin{remark}\label{r_exa}
\rm 
{\bf (i)} 
The fact that the parameters $\alpha$ and $\gamma$ are distinct can also be obtained as a consequence of Corollary \ref{c_cont} 
--- indeed, the parameter $\gamma$ can, by its definition, take only finitely many rational values. 

\smallskip

{\bf (ii) } 
The parameters $\alpha$ and $w$ are distinct. Indeed, 
let $\xi_t = t(e_1\otimes e_1) + \sqrt{1-t^2} (e_2\otimes e_2)$ in $\bb{C}^2\otimes \bb{C}^2$, $t\in [\frac{1}{\sqrt{2}},1]$. 
	By Proposition \ref{p_sep_alpha}, $\alpha(E_{\xi_t})=\frac{1}{2}(t+\sqrt{1-t^2})^2=\frac{1}{2}+t\sqrt{1-t^2}$. 
	On the other hand, Lemma \ref{p_w} implies that  $w(\xi_t) = \frac{1}{2t^2}$.  
\end{remark}

We finish this section with an observation about the parameters $\alpha$ and $\gamma$ in the case 
where $n = m = 2$.

\begin{proposition}\label{p_22}
    Let $E$ be a projection in $M_2\otimes M_2$ and $\xi$ be a unit vector in $\bb{C}^2\otimes\bb{C}^2$. Then
    \begin{itemize}
 \item[(i)] 
 $\alpha(E)=1$ if and only if $E(\mathbb C^2\otimes\mathbb C^2)$ contains  a maximally entangled vector;
 
 \item[(ii)] 
$\gamma(E_\xi) = \begin{cases} 1 &  \textup{ if  } \xi \textup{ is entangled;} \\ \frac{1}{2} & \textup{ if  } \xi \textup{ is separable.} \end{cases}$
\end{itemize}    
\end{proposition}

\begin{proof}
(i)    
Let $W=E(\mathbb C^2\otimes\mathbb C^2)$. If $W$ contains a maximally entangled unit vector $\xi\in \mathbb C^2\otimes\mathbb C^2$ then, 
    by Theorem \ref{p_sep}, $\alpha(E)\geq\alpha(E_{\xi})=1$ and hence $\alpha(E)=1$.
    
    Assume now $\alpha(E)=1$. 
    Then there exists $\sigma\in \bist(\tr,\tr)$ such that $\sigma(E)=1$, which is equivalent to  
    $\tr(A_\sigma(I-E))=0$ and hence, by the faithfulness of the trace, to 
    $EA_\sigma E=A_\sigma$ (indeed, our assumption yields that  $\tr (E^\perp A_\sigma^{\frac{1}{2}} A_\sigma^{\frac{1}{2}} E^\perp)=0$, so further $A_\sigma^{\frac{1}{2}} E^\perp=0$). We may assume that $\sigma$ is an extreme point. In fact, if $\sigma=\sum_{i=1}^n\lambda_i\sigma_i$ is a convex combination of states in $\bist(\tr,\tr)$, then 
    $\sigma=\sum_{i=1}^n\lambda_iE\sigma_iE$ and 
    $1=\sigma(1) = \sum_{i=1}^n\lambda_i\sigma_i(E)$, showing that 
    $\sigma_i(E) = 1$ for all $i=1,\ldots, n$.  

    Since $\sigma$ is now assumed an extreme point, 
    by \cite{kummer} the corresponding unital quantum channel $\Phi_{A_{\sigma/2}}$ 
    is given by a unitary conjugation. Thus there exists a unitary $U$ such that 
    $$\frac{1}{2}A_\sigma = \left[U\epsilon_{i,j}U^*\right]_{i,j=1}^2 = \left[(Ue_i)(Ue_j)^*\right]_{i,j=1}^2.$$ 
    Since $\sigma$ is supported on $E$ and $\frac{1}{4}A_\sigma$ is a projection,  
    $$\frac{1}{2}\left[(Ue_i)(Ue_j)^*\right]_{i,j=1}^2\leq E.$$ 
    But $\frac{1}{2}[(Ue_i)(Ue_j)^*]_{i,j=1}^2$ is the rank one projection of the maximally entangled vector
    $\frac{1}{\sqrt{2}}(Ue_1\otimes e_1 + Ue_2\otimes e_2)$, and the claim is proved. 
    
(ii)
If $\xi$ is separable then Theorem \ref{p_sep} (i) implies that $\gamma(E_{\xi}) = \frac{1}{2}$. 
Suppose that $\xi=\lambda_1 e_1 \otimes f_1 + \lambda_2 e_2 \otimes f_2$ 
in its Schmidt decomposition, and assume, by way of contradiction, that $\lambda_1 \lambda_2 \neq 0$. 
Let $p,q\in M_2$ be projections 
such that $E_\xi \leq p \otimes 1 \vee 1 \otimes q$; note that 
the latter condition is equivalent to the requirement $(p^\perp \otimes q^\perp)\xi = 0$. 
Suppose that $\tr (p)+\tr (q)<1$, in other words, that  $\tr (p^\perp)+\tr (q^\perp)>1$. 
This forces one of the projections, say $p^\perp$, to be equal $I$. 
But then $0 = (1 \otimes q^\perp) \xi = \lambda_1 e_1\otimes q^\perp f_1 + \lambda_2 e_2 \otimes q^\perp f_2$, 
and hence $0 = q^\perp f_1 = q^\perp f_2$, implying that $q^\perp=0$
and contradicts the assumotion that  $\tr (p)+\tr (q) < 1$.
\end{proof}


We note that the proof of Proposition \ref{p_22} uses the fact that the extreme points of the 
set of all unital quantum channels on $M_2$ are the unitary conjugation channels. It was 
proved in \cite{ohno} (and attributed to Arveson therein) that 
this is not true for $M_n$ with $n \geq 3$. It would be of interest to know if, nevertheless, 
Proposition \ref{p_22} remains valid in dimensions higher than two. 

\smallskip

\noindent {\bf Acknowledgments. } 
A.S. was  partially supported by the National Science Center (NCN) grant no. 2020/39/I/ST1/01566.
I.T. was supported by NSF Grant 2115071.
 L.T. would like to thank the Wenner-Gren Foundation which supported the visit of I.T. to Gothenburg. I.T and L.T acknowledge the support and hospitality at the Institute of Mathematics of the Polish Academy of Sciences during their visit in 2022. We thank the referee for a careful reading of our paper.

\end{document}